\documentclass[Threeside,11pt,tikz,border=3mm]{article}
\usepackage{etoolbox}
\makeatletter
\patchcmd{\thebibliography}
{\list}
{\small          
	\setstretch{1.2}
	\list}
{}{}
\makeatother

\usepackage{setspace} 
\usepackage{tikz}
\usepackage{tikz-cd}
\usetikzlibrary{arrows.meta}
\usepackage{subcaption}
\usepackage [latin1]{inputenc}
\usepackage[english]{babel}
\usepackage{amsmath,amscd}
\usepackage{amsthm}
\usepackage{algorithm}
 \usepackage{adjustbox}
\usepackage{cite}
\usepackage{amsfonts}
\usepackage{amssymb}
\usepackage{mathrsfs}
\usepackage{graphicx}
\usepackage{colortbl,dcolumn}
\usepackage{marvosym}
\usepackage{ifsym}
\usepackage{paralist}
\usepackage{xcolor}
\usepackage{array,color}
\usepackage{geometry}
\usepackage[colorlinks,
citecolor=blue,
linkcolor=blue,
backref=page
allcolors=black]{hyperref} 
\usetikzlibrary{shapes.geometric,arrows.meta,calc}
\usetikzlibrary{arrows, decorations.markings}
\definecolor{mydeepgreen}{HTML}{017A79}
\graphicspath{{figures/}}
\allowdisplaybreaks
       \newtheorem{lemma}{\bf Lemma}[section]
       \newtheorem{theorem}{\bf Theorem}[section]

       \newtheorem{remark}{\bf Remark}[section]

       \numberwithin{equation}{section}

\usepackage{enumitem}
\begin{document}
\title{{\sl Exponential convergence can happen in weighted Birkhoff averages via quasi-periodicity with arbitrary nonresonance and low regularity}}
\author{Zhicheng Tong $^{\mathcal{z}}$}

\renewcommand{\thefootnote}{}
\footnotetext{\hspace*{-6mm}
\begin{tabular}{l   l}	$^\mathcal{z}$~School of Mathematics, Jilin University, Changchun 130012, P. R.  China. \url{tongzc25@jlu.edu.cn}
\end{tabular}}

\date{}
\maketitle

\begin{abstract}
Since Krengel's work \cite{MR0510630} in 1978, it has been widely known that no effective rate of convergence exists in the ergodic theorem. For toral translations, however, by choosing appropriate weights one can accelerate the convergence of ergodic averages to an exponential rate, but this intuitively requires both highly nonresonant frequencies and very regular observables. In this paper, we uncover a new phenomenon: even for any given nonresonant frequency, there exists a non-trivial family of weights and observables of low regularity such that the weighted Birkhoff averages along quasi-periodic orbits converge at a quantitative, uniform, and exponential rate. This not only yields a finer understanding of the deep interaction between nonresonance and regularity in ergodic theory, but also stands as a weighted counterpart to a Yoccoz-type result \cite{MR604672,MR1367354}.\vspace{2mm}
\end{abstract}


\section{Introduction}
\subsection{Background and motivation}
\renewcommand{\thefootnote}{\fnsymbol{footnote}}
A fundamental problem in dynamical systems is determining the rate of convergence in the well-known ergodic theorem. However, in 1978, Krengel \cite{MR0510630} demonstrated that an effective rate of convergence cannot be expected in general: there exist cases in which the Birkhoff averages converge arbitrarily slowly.   Similar results were subsequently obtained by del Junco--Rosenblatt \cite{MR553340} and Kakutani--Petersen \cite{MR612450}, and recent advances were presented by Ryzhikov \cite{MR4602432,Ryzhikov2,Ryzhikov3,Ryzhikov4}. In 1996, Kachurovski\u{\i} \cite{MR1422228} established that for any non-trivial (non-constant) observable the $ N $-step Birkhoff average converges no faster than $ \mathcal{O} (N^{-1} ) $, and a host of subsequent advancements were obtained by  Kachurovski\u{\i}--Podvigin  \cite{MR3643963,MR3981324} and  Podvigin \cite{MR4440287,Podvigin2,MR4767798}. Collectively, these results reveal that, \textit{in general}, the ergodic theorem exhibits an intrinsic---indeed, unimprovable---slowness of convergence. Even for the \textit{concrete} setting of toral rotations, Yoccoz \cite{MR604672,MR1367354} showed that for every analytic observable there exist strongly Liouvillean (non-Diophantine) frequency vectors forcing the Birkhoff averages to converge arbitrarily slowly. 
	
	All of the previous observations immediately raise a pivotal question: can the classical Birkhoff average be accelerated by appropriate weighting?  In general, when the weights are uniform (i.e.,  independent of the time scale), weighted Birkhoff averages appear incapable of rapid convergence, see, for instance, Bellow--Losert \cite{MR773063}, Bellow--Jones--Rosenblatt \cite{MR1170516},  Durand--Schneider \cite{MR1906483},  Lin--Weber \cite{MR2308143},    Fan--Schmeling \cite{MR3897972},  Fan \cite{MR3928616}  and the references therein.  Attention therefore shifted to non-uniform weights---those that depend explicitly on the time scale.

In the 1990s, Laskar \cite{MR1234445,MR1222935,MR1720890} devised such a weighting scheme.  Numerically, it produced surprising \textit{exponential} acceleration for quasi-periodic systems, yet no proof was provided\footnote{Although Laskar did not provide a convergence analysis for weighted Birkhoff averages, a rigorous theoretical justification for his frequency map analysis was established in \cite{MR1720890}.}, nor was its scope explored beyond a handful of systems.  Despite this theoretical vacuum, Laskar's method and its variants became a workhorse in computational dynamics, a status they retain to this day. A relatively comprehensive overview is provided in \cite{TL-DP}. A rigorous proof for weighted Birkhoff averages came in 2017, when Das--Saiki--Sander--Yorke \cite{MR3718733} (see also Das--Yorke \cite{MR3755876}) established the first arbitrary polynomial convergence rate in the  discrete-time quasi-periodic setting.  In 2023, Duignan--Meiss \cite{MR4582163} extended the result to the continuous-time case. Subsequently, the present author and his collaborator 	\cite{MR4768308,TLdecaying,TL-DP,TLOPT}  established relatively comprehensive results involving almost-periodicity and exponential convergence.

Most of these rapid convergence results, however, still impose strong nonresonance (e.g., Diophantine condition) and high regularity (e.g., at least $ C^m $ with $ m $ sufficiently large). From a dynamical viewpoint, this is somewhat  coarse and merely broadly sufficient. This omission leaves the delicate interplay between nonresonance and regularity unexplored---precisely the issue at the heart of dynamical systems. Comparable questions have been studied in depth within many dynamical theories including KAM theory and linearization theory\footnote{But these differ from the intrinsic interplay between nonresonance and regularity that we seek in weighted Birkhoff averages, owing to the effect of infinitely many iterations.}; see, for instance, Brjuno \cite{MR377192}, Katznelson \cite{MR581808}, Yoccoz \cite{MR929279}, Salamon \cite{MR2111297}, Khanin--Teplinsky \cite{MR2545684}, Cheng--Wang \cite{MR3061774}, Tong--Li \cite{TLCCM} and the references therein.

\textit{This paper is therefore motivated by the desire to investigate, in depth, the interaction between nonresonance and regularity in the aforementioned Laskar-type results for quasi-periodic dynamical systems.} By way of contrast, to achieve the  fastest possible rate of convergence $ \mathcal{O} (N^{-1} ) $ in the unweighted case, a cohomological equation argument \cite{MR2039503} asserts that for any given nonresonant frequency it is sufficient to require rapid decay of the observable's Fourier coefficients near ``almost-rational'' modes. In the non-uniformly weighted case in which we are more concerned, under the hypotheses employed in \cite{MR3718733,MR3755876,MR4582163,MR4768308} the same philosophy yields arbitrary polynomial convergence, yet this point went unnoticed in earlier studies. \textit{Exponential convergence, however, has remained an open and intrinsically difficult question,} while Theorem \ref{RET1} in this paper fills the gap. To be more precise, our main discovery is that for any given nonresonant frequency, there exists a family of weights (close in spirit yet distinct from Laskar's original choice) and observables of low regularity, such that the weighted Birkhoff averages converge to the spatial averages at a uniform, quantitative, and exponential rate.  In earlier work, even when the frequency vector satisfied a Diophantine condition, analyticity or Gevrey-class regularity was typically imposed to ensure exponential convergence in a universal sense.  Therefore, from a perspective of   dynamical systems, this new contribution provides a quantitative characterization of how nonresonance and regularity interact, and it can be regarded as a counterpart to a Yoccoz-type  result \cite{MR604672,MR1367354}, in the weighted sense.

\subsection{Main result}
We start with some notation. For $ d \in \mathbb{N}^+ $, let $\mathbb T^d=\mathbb R^d/\mathbb Z^d$ denote the standard $d$-dimensional torus equipped with the normalized Lebesgue measure $\mu$.  A mapping ${\mathscr{T}}_\alpha:\mathbb{T}^d\to\mathbb{T}^d$ is called a \textit{toral translation} with frequency vector $\alpha\in\mathbb{R}^d$ if it is defined by ${\mathscr{T}}_\alpha(\theta)=\theta+\alpha \pmod{\mathbb{Z}^d} $ for all $ \theta \in \mathbb{T}^d $. In this context, the frequency vector $\alpha$ is said to be \textit{nonresonant} (or rationally independent) if $\langle k,\alpha \rangle:=\sum_{j=1}^dk_j\alpha_j\notin \mathbb{Z}$ for all $k\in\mathbb{Z}^d\setminus\{0\}$.

	Let $(\mathscr{B},\|\cdot\|_{\mathscr{B}})$ be a Banach space. For brevity, we denote it by $\mathscr{B}$ when no confusion arises. A  mapping $g:\mathbb{R}\to \mathscr B$ is said to be \textit{quasi-periodic} if there exist a continuous mapping  $F:\mathbb{T}^d\to \mathscr B$ and a nonresonant frequency vector $\alpha\in\mathbb{R}^d$ such that $g(t)=F(\alpha t)$ for all $t\in\mathbb{R}$. In particular, in the discrete-time setting where $t=n\in\mathbb{Z}$, the expression $F(\alpha n)$ can be rewritten in the form of a toral translation $F({\mathscr{T}}_\alpha^n(0))$. For the sake of clarity, we shall adopt the latter notation throughout this paper. Furthermore, one may also consider dynamical systems that are smoothly conjugate to quasi-periodic ones, as depicted in Figure \ref{FIG1} (for the case $d=1$). We remark in advance that although the main result of this paper, namely Theorem \ref{RET1} below, is formulated in terms of toral translations, it can be directly applied to the aforementioned systems by virtue of Figure \ref{FIG1}\footnote{For instance, via the previously mentioned KAM theory or linearization theory.}, thus yielding broad applicability.
	
	Next, we characterize the regularity of observables on the torus. Consider a continuous mapping $f:\mathbb{T}^d\to \mathscr{B}$ admitting the Fourier expansion $f(x)=\sum_{k \in \mathbb{Z}^d} f_k e^{2\pi i \langle k,x \rangle}$. We define the space $A_{\mathscr{B}}(\mathbb{T}^d)$ as
	\[A_{\mathscr{B}}(\mathbb{T}^d):=\left\{ f: \sum_{k \in \mathbb{Z}^d} \|f_k\|_{\mathscr{B}} < +\infty \right\}.\]
	Clearly, the regularity of observables in $A_{\mathscr{B}}(\mathbb{T}^d)$ is only marginally better than mere continuity. Consequently, throughout this paper, the term ``\textit{low regularity}'' exclusively refers to the regularity characterized by the space $A_{\mathscr{B}}(\mathbb{T}^d)$. The most elementary case for observables occurs when $d=1$ and $(\mathscr{B},\|\cdot\|_{\mathscr{B}})=(\mathbb{R},|\cdot|)$ ($\mathscr{B}=\mathbb{R} $ for brevity), which allows for a more intuitive investigation of regularity. In this setting, for $0<a<1$, the classical H\"older space $C^a(\mathbb{T})$ is defined as
	\[C^a(\mathbb{T}):=\left\{ f:\sup_{x,y \in \mathbb{T}, x \ne y} \frac{|f(x) - f(y)|}{|x - y|^a} < +\infty \right\}.\]
	We will explore the relationship between these two notions of regularity more thoroughly in Section \ref{SUB1}.

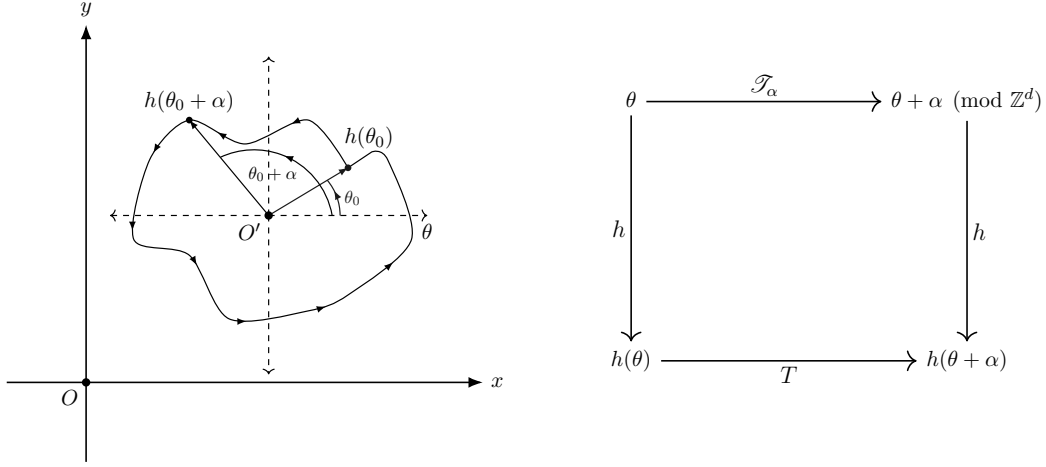
\begin{figure}[htbp]
	\centering 
	\begin{adjustbox}{scale=0.75}
		\begin{minipage}[c]{0.60\textwidth} 
			\centering 
			\begin{tikzpicture}[scale=1.4] 
				\draw[->,  thick, -{Latex[scale=1.1]}] (-1cm,0cm) -- (5cm,0cm) node[right]{\(x\)}; 
				\draw[->,  thick, -{Latex[scale=1.1]}] (0cm,-1cm) -- (0cm,4.5cm) node[above]{\(y\)}; 
				
				\fill (0cm,0cm) circle (1.5pt) node[below left]{\(O\)};

				\draw[fill=white, semithick, decoration={markings,
					mark=between positions 0.1 and 0.9 step 0.1 with {\arrow{latex}},
				},postaction={decorate}
				] plot[smooth, tension=0.5] coordinates {
					(3cm+0.3cm,2.4cm+0.3cm) 
					(2.5cm+0.3cm,3cm+0.3cm) 
					(1.7cm+0.3cm,2.7cm+0.3cm) 
					(1cm+0.3cm,3cm+0.3cm) 
					(0.5cm+0.3cm,2.5cm+0.3cm) 
					(0.3cm+0.3cm,1.5cm+0.3cm) 
					(1cm+0.3cm,1.3cm+0.3cm) 
					(1.5cm+0.3cm,0.5cm+0.3cm) 
					(2.5cm+0.3cm,0.6cm+0.3cm) 
					(3cm+0.3cm,0.8cm+0.3cm) 
					(3.8cm+0.3cm,1.5cm+0.3cm) 
					(3.5cm+0.3cm,2.5cm+0.3cm)
					(3.3cm+0.3cm,2.6cm+0.3cm)
				} -- cycle;
				
				
				\fill[black!90] (3cm+0.3cm,2.4cm+0.3cm) circle (1.2pt) node[above right, yshift=0.2cm, xshift=-0.2cm, black] {\(h(\theta_0)\)};
				
				\fill[black]  (1cm+0.3cm,3cm+0.3cm) circle (1.2pt) node[above,black] {\(h(\theta_0+\alpha)\)};
				
				\draw[dashed, ->, black, semithick] (2cm+0.3cm,1.8cm+0.3cm) -- (4.0cm+0.3cm,1.8cm+0.3cm) node[below] {\(\theta\)};
				
				\draw[dashed, ->, black, semithick] (2cm+0.3cm,1.8cm+0.3cm) -- (2cm+0.3cm,3.8cm+0.3cm);
				
				\draw[dashed, <-, black, semithick] (0cm+0.3cm,1.8cm+0.3cm) -- (2.0cm+0.3cm,1.8cm+0.3cm);
				
				\draw[dashed, ->, black, semithick] (2cm+0.3cm,1.8cm+0.3cm) -- (2cm+0.3cm,-0.2cm+0.3cm);
				
				\draw[-latex,black!90,semithick] (2cm+0.3cm,1.8cm+0.3cm) -- (3cm+0.3cm,2.4cm+0.3cm);
				
				\draw[-latex,black,semithick] (2cm+0.3cm,1.8cm+0.3cm) -- (1cm+0.3cm,3cm+0.3cm);
				
				\tikzset{
					arc arrow/.style = {
						decoration = {
							markings,
							mark = at position 0.68 with {\arrow{latex}}
						},
						postaction = decorate
					}
				}
				
				\draw[black!90, semithick, arc arrow] (2.9cm+0.3cm,1.8cm+0.3cm) arc (0:40:0.69cm)
				node[midway, below right, yshift=0.25cm] {\scalebox{0.8}{\(\theta_0\)}};

				\tikzset{arc arrow/.style = {
						decoration = {markings,
							mark = at position 0.55 with {\arrow{latex}}
						},
						postaction = decorate
				}}
				
				\draw[black, semithick, arc arrow] (1.8cm+0.3cm,1.8cm+0.3cm) ++(0:1cm) arc (10:115:1cm)
				node[midway, below left, xshift=0.25cm] {\scalebox{0.8}{\(\theta_0+\alpha\)}};
				

				\fill (2cm+0.3cm,1.8cm+0.3cm) circle (1.5pt) node[below left] {$O'$};
			\end{tikzpicture}
		\end{minipage}
		\hfill 
		\begin{minipage}[c]{0.60\textwidth} 
			\centering 
			\begin{tikzcd}[row sep=3.9 cm, column sep=3.9 cm, arrows={thick}]
				{\theta}  \arrow[r, " {\mathscr{T}}_\alpha"{font=\large}] \arrow[d, "h"'{font=\large}] & \theta+\alpha \pmod{\mathbb{Z}^d}  \arrow[d, "h"{font=\large}] \\
				h(\theta) \arrow[r, "T"'{font=\large}] & h(\theta+\alpha)
			\end{tikzcd}
		\end{minipage}
	\end{adjustbox}	\caption{Geometric properties and the commutative diagram of the conjugacy $T \circ h(\theta)=h \circ {\mathscr{T}}_\alpha(\theta)$ for $\theta\in\mathbb{T}^d$, where $ T: X \to X $ and $h:\mathbb{T}^d \to X$ are smooth diffeomorphisms.}\label{FIG1}
\end{figure}

Our main result is now stated as follows:

\renewcommand{\thefootnote}{\fnsymbol{footnote}}
\begin{theorem}\label{RET1}
Consider a toral translation $\mathscr{T}_{\alpha}$ on the $d$-dimensional torus $\mathbb T^{d}$ ($d \geqslant 1$). Suppose that the frequency vector $\alpha \in \mathbb R^{d}$ is nonresonant. Then, for any $\sigma \in (0,1)$, there exist a family of normalized weights $\{w_{n,N}\}_{n=0}^{N-1}$\footnote{That is, $ \sum\nolimits_{n = 0}^{N - 1} {{w_{n,N}}}  = 1 $.} and a family of observables $f \in A_{\mathscr B}(\mathbb T^{d})$ of low regularity such that
			\[\mathop {\lim \sup }\limits_{N \to  + \infty } {\frac{1}{{{N^\sigma }}}}\log \left( {\mathop {\sup }\limits_{x \in \mathbb{T}^d} {{\left\| {\sum\limits_{n = 0}^{N - 1} {{w_{n,N}}f\left( {{{\mathscr{T}}_\alpha^n}x} \right)}  - \int_{\mathbb{T}^d} {fd\mu } } \right\|}_{\mathscr{B}}}} \right) < 0.\]
			In particular, when $d=1$ and $\mathscr B=\mathbb R$, these observables may fail to belong to $C^{a}(\mathbb T)$ for any $a \in (0,1)$\footnote{Indeed, one could construct examples exhibiting an even weaker modulus of continuity that satisfies certain Dini-type integrability conditions in the spirit of \cite{TLCCM}.}.
\end{theorem}
\begin{remark}It should be emphasized that for almost every frequency vector---even those satisfying a strong Diophantine condition---exponential convergence of the type stated in Theorem \ref{RET1} typically demands extremely high regularity for observables, such as analyticity \cite{TLdecaying,TL-DP}.
	
 Given that the nonresonance of the frequency vector can be prescribed arbitrarily and that the observables are allowed to be so irregular, it is somewhat surprising that Theorem \ref{RET1} still achieves such rapid exponential convergence.
 \end{remark}

We conclude this section with a few additional remarks.
\renewcommand{\thefootnote}{\fnsymbol{footnote}}
	\begin{enumerate}[label=(\roman*)] 
\item \label{item1} Theorem \ref{RET1} constructs a family of low-regularity observables  adapted to the prescribed nonresonant frequency,  demonstrating that exponential convergence can indeed happen in this setting. Section \ref{RESEC2} presents explicit, non-trivial constructions in detail. Generic observables with low regularity, however, fail to exhibit such rapid convergence: the convergence is intrinsically slow, as shown by the optimal lower bounds obtained in \cite{TLOPT} in a universal manner\footnote{In the sense of full Lebesgue measure.}. 

\item 
In contrast to \ref{item1},  establishing upper bounds on the convergence rate for generally prescribed weak nonresonance  and low regularity \textit{without} the restrictions imposed in \cite{MR3718733,MR3755876,MR4582163,MR4768308} is indeed nontrivial (this requires different techniques). A rigorous approach to this general case will be presented in another paper.

\item 
 Theorem \ref{RET1} focuses on quasi-periodic systems. However, the almost-periodic setting demands more intricate spatial structures and nonresonance conditions 	
\cite{MR2180074,MR4201442}, together with extra analytical and combinatorial techniques \cite{MR4768308,TL-DP}. Nevertheless, an almost-periodic analogue of Theorem \ref{RET1} exhibiting \textit{slower} exponential convergence can be established using the novel methods developed herein.
		\end{enumerate}

\section{Proof of Theorem \ref{RET1}}\label{RESEC2}
Theorem \ref{RET1} aims to characterize, for every prescribed nonresonant frequency vector, which non-trivial observables of \textit{relative regularity} guarantee that the weighted Birkhoff averages converge at a quantitative, uniform, and exponential rate. Here, ``uniform'' refers not only to the initial points but also to the fact that the same rate is shared across the entire family of observables. We also remark that the Banach-valued nature of the observables is not essential, as the proof remains identical to that in the real-valued case.

We first briefly outline the heuristic ideas behind the proof of Theorem \ref{RET1}. As mentioned in the Introduction, the arguments in \cite{MR3718733,MR3755876,MR4582163,MR4768308} directly imply Theorem \ref{RET1} with an arbitrary polynomial convergence rate, although this was not explicitly stated. The reason that exponential convergence for weighted Birkhoff averages could be established in \cite{MR4768308,TLdecaying,TL-DP} is that the authors obtained exponential asymptotics for a single weighted harmonic oscillator and introduced novel techniques such as truncation. Notably, \cite{TLdecaying,TL-DP} provided a quantitative characterization of such exponential rates. Consequently, by appropriately selecting harmonic oscillators to construct a family of observables, it becomes possible to achieve a balance between arbitrary nonresonance and low  regularity, thereby yielding the quantitative exponential convergence asserted in Theorem \ref{RET1}.

 To keep the argument readable, we divide the proof into three steps. Step 1 (Section \ref{SUB1}): Under arbitrary nonresonance and low regularity, we provide an explicit construction.  This involves a decomposition of $ \mathbb{Z}^d\setminus{\{0\}} $ and the construction of observables adapted to each subset. Step 2 (Section \ref{SUB2}):  A local-minimization technique yields a key lemma that establishes the quantitative exponential asymptotics for a single weighted harmonic oscillator depending on the nonresonance. Step 3 (Section \ref{SUB3}):  We refine the original decomposition so that the convergence rate is dominated by three independent sums. A noteworthy observation is that, with a suitable choice of parameters, these quantities admit a uniform upper bound. Balancing nonresonance against regularity, we obtain the final quantitative uniform exponential convergence for the weighted Birkhoff averages.

\subsection{Explicit constructions via arbitrary nonresonance and low regularity}\label{SUB1}
For any $ 0<\sigma<1 $, we first select normalized weights $\{w_{n,N}\}_{n=0}^{N-1}$ of the form $ {w_{n,N}} :={Z_N^{ - 1}} {{\widetilde{w}}_{p,q}}\left( {n/N} \right) $, with the function $ {\widetilde{w}}_{p,q} $ explicitly defined by
\begin{equation}\notag
	{{\widetilde{w}}}_{p,q}\left( x \right):= \left\{ \begin{array}{ll}
		\exp \left( { - {x^{ - p}}{{\left( {1 - x} \right)}^{ - q}}} \right),&x \in \left( {0,1} \right),  \hfill \\
		0,&x \notin \left(0,1\right), \hfill \\
	\end{array}  \right.
\end{equation}
where $Z_N:={\sum\nolimits_{j = 0}^{N - 1} {{{\widetilde w}_{p,q}}\left( {j/N} \right)} }$, and the positive constants $ p,q $ satisfy
\[\min \left\{ {p,q} \right\} > {\left( {{\sigma ^{ - 1}} - 1} \right)^{ - 1}}.\]
The function $ {\widetilde{w}}_{p,q} (x)$ belongs to $C_0^\infty([0,1])  $, with its graph resembling a bell curve or an extremely compressed Gaussian density. Additionally, the following key lemma  describes the asymptotic behavior of its higher-order derivatives:
 \begin{lemma}[Lemma 4.1 in \cite{TLdecaying}]\label{YLZY}
 	There exists $\beta_{p,q}:= {1 + \left({\min \left\{ {p,q} \right\}}\right)^{-1}} $ and some $ \lambda  = \lambda \left( {p,q} \right) > 1 $, such that for any $ m \in \mathbb{N}^+ $,
 	\[{\left\| {{D^m}{{\widetilde w}_{p,q}}\left( x \right)} \right\|_{{L^1}\left( {0,1} \right)}} \leqslant {\lambda ^m}{m^{\beta_{p,q} m}}.\]	
 \end{lemma}

We point out that we can consider other weighting functions in $C_0^\infty([0,1])  $ satisfying the property stated in Lemma \ref{YLZY}. Construction of such functions is a question.

Next, for a fixed nonresonant frequency $\alpha \in \mathbb R^{d}$ and any $\epsilon \in (0,1)$, we decompose $\mathbb Z^{d} \setminus \{0\}$ into two parts:
\begin{align*}
\Upsilon _\alpha ^ \leqslant : &= \left\{ {k \in {\mathbb{Z}^d}\setminus \left\{ 0 \right\}:{\rm dist}\left( {\left\langle {k,\alpha } \right\rangle ,\mathbb{Z}} \right) \leqslant \epsilon} \right\},\\
\Upsilon _\alpha ^ > :& = \left\{ {k \in {\mathbb{Z}^d}\setminus \left\{ 0 \right\}:{\rm dist}\left( {\left\langle {k,\alpha } \right\rangle ,\mathbb{Z}} \right) > \epsilon} \right\}.
\end{align*}
Here, $ {\rm dist}\left( {\left\langle {k,\alpha } \right\rangle ,\mathbb{Z}} \right): = {\inf _{n \in \mathbb{Z}}}\left| {\left\langle {k,\alpha } \right\rangle  - n} \right|\ne 0 $.
 
Now, we are going to construct a family of observables $f \in A_{\mathscr B}(\mathbb T^{d})$ by letting the Fourier coefficients $f_{k}$ of $f$ free on $\Upsilon_{\alpha}^{>}$ and putting relatively regular conditions on $\Upsilon_{{\alpha}}^{\leqslant}$:
\begin{equation}\label{FKDY}
	{\left\| {{f_k}} \right\|_{\mathscr{B}}} \leqslant \exp \left( { - \frac{1}{{{\rm dist}^{\kappa}\left( {\left\langle {k,\alpha } \right\rangle ,\mathbb{Z}} \right)}}} \right)g\left( {{{\left\| k \right\|}_{{\ell ^1}}}} \right),\quad \kappa  > {\left( {1 - {\beta _{p,q}}\sigma } \right)^{ - 1}}\sigma ,\quad k \in \Upsilon _\alpha ^ \leqslant ,
\end{equation}
		 where $g$ is a non-decreasing continuous function on $[1, +\infty)$ satisfying the integrability condition 
\begin{equation}\label{GTJ}
	\int_1^{ + \infty } {{x^{d - 1}}g\left( x \right)dx}  <  + \infty ,
\end{equation}
 and $ {\left\| k \right\|_{{\ell ^1}}}: = \sum\nolimits_{j = 1}^d {\left| {{k_j}} \right|}  $ denotes the standard $ \ell ^1 $-norm. A simple example is given by $g(x) = x^{-d}\log^{-b}(1+x)$, where $b > 1$ is arbitrary. Note that the regularity of $f$ need not exceed that of $A_{\mathscr B}(\mathbb T^{d})$.

To illustrate more clearly that our constructions can exhibit low regularity, we consider the simplest case where $d=1$ and $\mathscr{B} = \mathbb{R}$. We define an observable $\varphi: \mathbb{T} \to \mathbb{R}$ by
 		\[\varphi(x) := \sum_{2 \leqslant k \in \Upsilon_\alpha^>} \frac{\cos(2\pi kx)}{k\log^2 k}, \quad x \in \mathbb{T}.\]
 		Here, we set $0 < \epsilon < 8^{-1}$ in the construction of the set $\Upsilon_\alpha^>$. It is evident that $ \varphi  \in {A_{\mathscr{B}}}(\mathbb T) $. For $0 < h \ll 1$, a direct calculation yields
 		\[\varphi(0) - \varphi(h) = 2\sum_{2 \leqslant k \in \Upsilon_\alpha^>} \frac{\sin^2(\pi kh)}{k\log^2 k}.\]
 		Now, consider the index set
 		\[\Theta_h := \left\{ k \in \mathbb{N}^+: 2 \leqslant k \in \Upsilon_\alpha^>,\ (4h)^{-1} \leqslant k \leqslant (2h)^{-1} \right\}.\]
 		Since $\alpha \in \mathbb{R}$ is nonresonant, Weyl's equidistribution theorem implies that the cardinality of $\Theta_h$ satisfies
 		\[\# {\Theta _h} \gtrsim \left( {1 - 2\epsilon } \right)\left( {{{\left( {2h} \right)}^{ - 1}} - {{\left( {4h} \right)}^{ - 1}}} \right) > {\left( {8h} \right)^{ - 1}}.\]
 		Furthermore, for $k \in \Theta_h$, we have $\sin^2(\pi kh) \geqslant \sin^2(4^{-1}\pi) = 2^{-1}$ and $k\log^2 k \leqslant (2h)^{-1}\log^2(2h)$. Therefore, for any given $0<a<1$, we have
 		\[|\varphi(0) - \varphi(h)| \geqslant 2\sum_{k \in \Theta_h} \frac{\sin^2(\pi kh)}{k\log^2 k}  \gtrsim  \frac{1}{4\log^2(2h)} \gg h^a\]
 		as $h \to 0^+$. This implies $\varphi \notin C^a(\mathbb{T})$, which completes the proof of the regularity part of Theorem \ref{RET1}.

Finally, utilizing the normalization  property of the weights $\{w_{n,N}\}_{n\in \mathbb{N}}$, we have that for $ x=\theta \in \mathbb{T}^d $,
\begin{align*}
\sum\limits_{n = 0}^{N - 1} {{w_{n,N}}f\left( {{{\mathscr{T}}_\alpha^n}x} \right)}  - \int_{\mathbb{T}^d} {fd\mu }  &= \sum\limits_{n = 0}^{N - 1} {{w_{n,N}}f\left( {x + n\alpha } \right)}  - \int_{{\mathbb{T}^d}} {fd\mu } \\
& = \sum\limits_{k \in {\mathbb{Z}^d}\setminus \left\{ 0 \right\}} {{f_k}{e^{2\pi i\left\langle {k,x} \right\rangle }}} S_{N,k}(\alpha),
\end{align*}
where $ 	S_{N,k}(\alpha):=\sum\nolimits_{n = 0}^{N - 1} {{w_{n,N}}{e^{2\pi ni\left\langle {k,\alpha } \right\rangle }}} $ is defined for brevity.
This implies that the uniform convergence rate of the weighted Birkhoff averages can be bounded by
\begin{equation}\label{FJ1}
	\mathop {\sup }\limits_{x \in \mathbb{T}^d} {\left\| {\sum\limits_{k \in {\mathbb{Z}^d}\setminus \left\{ 0 \right\}} {{f_k}{e^{2\pi i\left\langle {k,x} \right\rangle }}S_{N,k}(\alpha) } } \right\|_{\mathscr{B}}} \leqslant \sum\limits_{k \in {\mathbb{Z}^d}\setminus \left\{ 0 \right\}} {{{\left\| {{f_k}} \right\|}_{\mathscr{B}}}\left| S_{N,k}(\alpha) \right|} .
\end{equation}

\subsection{Asymptotics for a single weighted harmonic oscillator via nonresonance}\label{SUB2}
In this section we establish a key lemma asserting that, for any prescribed nonresonance, a single weighted harmonic oscillator exhibits quantitative exponential decay within a suitable asymptotic nonresonant scale. The proof relies primarily on a local-minimization technique coupled with a small divisor argument.

\begin{lemma}\label{GJGJ}
There exists some  $ c>0 $, such that for every $ k \in {\mathbb{Z}^d}\setminus \left\{ 0 \right\} $, every nonresonant $ \alpha \in \mathbb{R}^d $, and every $ N \in \mathbb{N}^+ $ satisfying $ N{\rm dist}\left( {\left\langle {k,\alpha } \right\rangle ,\mathbb{Z}} \right) \gg 1 $, we have
\[\left| S_{N,k}(\alpha) \right| \leqslant \exp \left( { - c{{\left( {N{\rm dist}\left( {\left\langle {k,\alpha } \right\rangle ,\mathbb{Z}} \right)} \right)}^{\frac{1}{{{\beta _{p,q}}}}}}} \right).\]
\end{lemma}
\begin{proof}
Recall that the weights $ \{w_{n,N}\}_{n \in \mathbb{N}} $ constructed in Section \ref{SUB1} are compactly supported.
Hence, assuming $\gamma:=\langle k,\alpha\rangle\ne 0$ throughout this section, we have  \[S_{N,k}(\alpha) =\sum\limits_{n = 0}^{N - 1} {{w_{n,N}}{e^{2\pi ni\left\langle {k,\alpha } \right\rangle }}}= \frac{1}{Z_N}\sum_{n=0}^{N-1}{{\widetilde w}_{p,q}}\left( {\frac{n}{N}} \right)e^{2\pi ni\gamma} = \frac{1}{Z_N}\sum_{n=-\infty}^{\infty}{{\widetilde w}_{p,q}}\left( {\frac{n}{N}} \right)e^{2\pi ni\gamma}.\]
Applying the Poisson summation formula \cite{MR1970295,MR3243734} to the function $x\mapsto \widetilde{w}_{p,q}\left(x/N\right)e^{2\pi xi\gamma}$, we obtain 
\[S_{N,k}(\alpha)=\frac{1}{Z_N}\sum_{n=-\infty}^{\infty}\int_{-\infty}^{\infty}e^{-2\pi n i  t}{{\widetilde w}_{p,q}}\left( {\frac{t}{N}} \right)e^{2\pi ti\gamma}dt.\]
Consequently, by the change of variables $t=Ns$, we have
\begin{equation}\label{LM221}
	\left|S_{N,k}(\alpha)\right|\lesssim \sum\limits_{n =  - \infty }^{ + \infty } {\left| {\int_0^1 {{{\widetilde w}_{p,q}}\left( s \right){e^{2\pi N\left( {\gamma  - n} \right)is}}ds} } \right|} ,
\end{equation}
where we have used the estimate
\[\frac{N}{Z_N} = {\left( {\sum\limits_{j = 0}^{N - 1} {\frac{1}{N}{{\widetilde w}_{p,q}}\left( {\frac{j}{N}} \right)} } \right)^{ - 1}} \lesssim 1,\]
which is guaranteed by the integrability of $\widetilde{w}_{p,q}$. In what follows, we  perform detailed estimates on the right-hand side of \eqref{LM221}.
  
  For fixed $ k \in {\mathbb{Z}^d}\setminus \left\{ 0 \right\} $, according to  the nonresonance  of $ \alpha\in \mathbb{R}^d $, there exists a unique $ n^* \in \mathbb{Z} $ such that $ {\left| \gamma  - {n^ * } \right|} = {\rm dist}\left( {\gamma ,\mathbb{Z}} \right) \in \left( {0,1/2} \right) $.    By integration by parts, it is evident that for any $\eta \ne 0$ and $\nu \in \mathbb{N}^+$,
  \begin{align}
   \left| {\int_0^1 {{{\widetilde w}_{p,q}}\left( s \right){e^{2\pi N\eta i s}}ds} } \right| &= \left| {{{\left( {2\pi iN\eta } \right)}^{ - \nu }}\int_0^1 {{D^\nu }{{\widetilde w}_{p,q}}\left( s \right){e^{2\pi N\eta i s}}ds} } \right|\notag \\
& \leqslant {\left| {2\pi N\eta } \right|^{ - \nu }}\int_0^1 {\left| {{D^\nu }{{\widetilde w}_{p,q}}\left( s \right)} \right|\left| {{e^{2\pi N\eta is}}} \right|ds}  \notag \\
   &= {\left| {2\pi N\eta } \right|^{ - \nu }}{\left\| {{D^\nu }{{\widetilde w}_{p,q}}\left( x \right)} \right\|_{{L^1}\left( {0,1} \right)}}. \label{FBJF}
  \end{align}
Next, utilizing \eqref{FBJF} (with $\eta =\gamma -n ^*$) and Lemma \ref{YLZY}, we obtain 
 \begin{align}
 	 \left| {\int_0^1 {{{\widetilde w}_{p,q}}\left( s \right){e^{2\pi N\left( {\gamma  - {n^ * }} \right)is}}ds} } \right| &\leqslant{\left( {2\pi N{\rm dist}\left( {\gamma ,\mathbb{Z}} \right)} \right)^{ - m}}{\left\| {{D^m}{{\widetilde w}_{p,q}}\left( x \right)} \right\|_{{L^1}\left( {0,1} \right)}}\notag \\
  & \leqslant {\left( {2\pi N{\rm dist}\left( {\gamma ,\mathbb{Z}} \right)} \right)^{ - m}}{\lambda ^m}{m^{{\beta _{p,q}}m}}\notag \\
 	& = \exp \left( {{\beta _{p,q}}m\log m + m\log \left( {\frac{\lambda }{{2\pi N{\rm dist}\left( {\gamma ,\mathbb{Z}} \right)}}} \right)} \right),\label{YLFJ1} 
 \end{align}
where $ 2 \leqslant m \in \mathbb{N}^+$ is required to satisfy the following asymptotic property\footnote{Here, the symbol ``$\sim$'' denotes asymptotic equivalence in the limiting process. In fact, $m$ in \eqref{mdingyi} can be taken as the integer part of the quantity on the right-hand side; the same applies to $m'$ below.}
\begin{equation}\label{mdingyi}
	m \sim \frac{1}{e}{\left( {\frac{\lambda }{{2\pi N{\rm dist}\left( {\gamma ,\mathbb{Z}} \right)}}} \right)^{ - \frac{1}{\beta_{p,q} }}}
\end{equation}
 to locally minimize \eqref{YLFJ1} by analyzing the monotonicity  with respect to $m$. Note that \eqref{mdingyi} is attainable, thanks to the assumption $ N{\rm dist}\left( {\gamma ,\mathbb{Z}} \right) \gg 1 $. This technique then allows us to establish the estimate for $ n^* $ by substituting \eqref{mdingyi} into \eqref{YLFJ1}: 
\begin{equation}\label{HHH}
	\left| {\int_0^1 {{{\widetilde w}_{p,q}}\left( s \right){e^{2\pi N\left( {\gamma  - {n^ * }} \right)is}}ds} } \right| \lesssim \exp \left( { - {c_1}{{\left( {N{{\rm dist}}\left( {\gamma ,\mathbb{Z}} \right)} \right)}^{\frac{1}{{{\beta _{p,q}}}}}}} \right),
\end{equation}
provided a universal constant $ 0 < {c_1} < {e^{ - 1}}\beta_{p,q} {\left( {2\pi {\lambda ^{ - 1}}} \right)^{{\beta_{p,q} ^{ - 1}}}} $.

As for all other indices $ n \ne {n^ * } $, we observe that a similar exponential upper bound applies. Define $ 2 \leqslant m' \in \mathbb{N}^+$ with
\[	m' \sim \frac{1}{e}{\left( {\frac{\lambda }{{\pi N}}} \right)^{ - \frac{1}{\beta_{p,q} }}}.\]
Therefore, utilizing \eqref{FBJF} (with $\eta =\gamma -n $) and
\[\left| {\gamma  - n} \right| \geqslant \left| {n - {n^ * }} \right| - \left| {\gamma - {n^ * }} \right| > \left| {n - {n^ * }} \right| - \frac{1}{2},\quad {n^ * } \ne n \in \mathbb{Z},\]
 we have 
\begin{align}
\sum\limits_{{n^ * } \ne n \in \mathbb{Z}} {\left| {\int_0^1 {{{\widetilde w}_{p,q}}\left( s \right){e^{2\pi N\left( {\gamma  - n} \right)is}}ds} } \right|}  
& \leqslant \sum\limits_{{n^ * } \ne n \in \mathbb{Z}} {{{\left( {2\pi N\left| {\gamma  - {{{n}}  }} \right|} \right)}^{ - m'}}{{\left\| {{D^{m'}}{{\widetilde w}_{p,q}}\left( x \right)} \right\|}_{{L^1}\left( {0,1} \right)}}} \notag \\
  & \leqslant \sum\limits_{{n^ * } \ne n \in \mathbb{Z}} {{{\left( {2\pi N\left( {\left| {n - {n^ * }} \right| - \frac{1}{2}} \right)} \right)}^{ - m'}}{{\left\| {{D^{m'}}{{\widetilde w}_{p,q}}\left( x \right)} \right\|}_{{L^1}\left( {0,1} \right)}}} \notag \\
   & \leqslant \left( {\sum\limits_{{n^ * } \ne n \in \mathbb{Z}} {{{\left( {2\left| {n - {n^ * }} \right| - 1} \right)}^{ - 2}}} } \right)\left( {{{\left( {\pi N} \right)}^{ - m'}}{{\left\| {{D^{m'}}{{\widetilde w}_{p,q}}\left( x \right)} \right\|}_{{L^1}\left( {0,1} \right)}}} \right)\notag.
\end{align}
Since
		\[\sum\limits_{{n^ * } \ne n \in \mathbb{Z}} {{{\left( {2\left| {n - {n^ * }} \right| - 1} \right)}^{ - 2}}}  = 2\sum\limits_{j = 0}^\infty  {\frac{1}{{{{\left( {2j + 1} \right)}^2}}}}  <  + \infty, \]
		we arrive at
		\begin{equation}\label{fjzh}
			\sum\limits_{{n^ * } \ne n \in \mathbb{Z}} {\left| {\int_0^1 {{{\widetilde w}_{p,q}}\left( s \right){e^{2\pi N\left( {\gamma  - n} \right)is}}ds} } \right|} \lesssim \exp \left( { - {c_2}{N^{\frac{1}{{{\beta _{p,q}}}}}}} \right)	,
		\end{equation}
		analogously to the analysis in  \eqref{HHH}, where $ 0 < {c_2} < {e^{ - 1}}\beta_{p,q} {\left( {\pi {\lambda ^{ - 1}}} \right)^{{\beta_{p,q} ^{ - 1}}}} $. 
 
 Finally, by choosing a sufficiently small universal constant $ 0<c<c_2 $ and combining \eqref{LM221}, \eqref{HHH}, and \eqref{fjzh} (noting that $ {\rm{dist}}\left( {\gamma ,\mathbb{Z}} \right) < 1 $), we complete the proof of Lemma \ref{GJGJ}:
 \[\left|S_{N,k}(\alpha)\right| \leqslant\exp \left( { - c{{\left( {N{\rm dist}\left( {\left\langle {k,\alpha } \right\rangle ,\mathbb{Z}} \right)} \right)}^{\frac{1}{{{\beta _{p,q}}}}}}} \right). \]
\end{proof}
As we shall see in Section \ref{SUB3}, Lemma \ref{GJGJ} will play a pivotal role in deriving the quantitative uniform exponential convergence rate for the weighted Birkhoff averages.

\subsection{Quantitative uniform exponential convergence of the weighted Birkhoff averages}\label{SUB3}
To weaken the impact of nonresonance (or of the small divisors that appear), we need to refine the decomposition of the set $ \Upsilon _\alpha ^ \leqslant  $. In fact, this involves a truncation technique depending on the time length $ N $. For $ N\in \mathbb{N}^+ $ sufficiently large and any $ {\kappa ^{ - 1}}\sigma  < \zeta  < 1 - {\beta _{p,q}}\sigma  $ (with the previous choices of $ \kappa,p $, and $ q $, the parameter $ \zeta $ is indeed well-defined), construct another two subsets $ \Lambda _{\alpha ,N}^ \leqslant ,\Lambda _{\alpha ,N}^ >  $ of $ \Upsilon _\alpha ^ \leqslant  $ as
\begin{align*}
	\Lambda _{\alpha ,N}^ \leqslant :& = \left\{ {k \in \Upsilon _\alpha ^ \leqslant :{\rm dist}\left( {\left\langle {k,\alpha } \right\rangle ,\mathbb{Z}} \right) \leqslant {N^{ - \zeta }}} \right\},\\
	\Lambda _{\alpha ,N}^ > : &= \left\{ {k \in \Upsilon _\alpha ^ \leqslant :{\rm dist}\left( {\left\langle {k,\alpha } \right\rangle ,\mathbb{Z}} \right) > {N^{ - \zeta }}} \right\}.
\end{align*}
They form a  decomposition of $ \Upsilon _\alpha ^ \leqslant  $ as $ N \gg 1 $, because $ \Lambda _{\alpha ,N}^ \leqslant  \cap \Lambda _{\alpha ,N}^ >  = \varnothing  $, and $ \Lambda _{\alpha ,N}^ \leqslant  \cup \Lambda _{\alpha ,N}^ >  = \Upsilon _\alpha ^ \leqslant  $ for $ N > {\epsilon ^{ - {\zeta ^{ - 1}}}} $. In what follows, we establish the corresponding estimates for the three cases $ k \in \Lambda _{\alpha ,N}^ \leqslant  $, $ k \in \Lambda _{\alpha ,N}^ >  $, and $ k \in \Upsilon _\alpha ^ > $.

For the sum over $\Lambda_{\alpha,N}^{\leqslant}$, we deliberately ignore the nonresonance and instead utilize the rapid decay of the Fourier coefficients together with their extra absolute summability in $ k $ to establish an exponential convergence rate. To be more precise, we obtain from \eqref{FKDY} and 
\[{\left| S_{N,k}(\alpha) \right|} \leqslant \sum\limits_{n = 0}^{N - 1} {\left| {{w_{n,N}}} \right|}  = \sum\limits_{n = 0}^{N - 1} {{w_{n,N}}}  = 1\]
 that
\begin{align}
\sum\limits_{k \in \Lambda _{\alpha ,N}^ \leqslant } {{{\left\| {{f_k}} \right\|}_{\mathscr{B}}}{\left| S_{N,k}(\alpha) \right|}}  &\leqslant \sum\limits_{{\rm dist}\left( {\left\langle {k,\alpha } \right\rangle ,\mathbb{Z}} \right) \leqslant {N^{ - \zeta }}} {{{\left\| {{f_k}} \right\|}_{\mathscr{B}}}} \notag \\
& \leqslant \sum\limits_{{\rm dist}\left( {\left\langle {k,\alpha } \right\rangle ,\mathbb{Z}} \right) \leqslant {N^{ - \zeta }}} {\exp \left( { - \frac{1}{{{\rm dist}^{\kappa}\left( {\left\langle {k,\alpha } \right\rangle ,\mathbb{Z}} \right)}}} \right)g\left( {{{\left\| k \right\|}_{{\ell ^1}}}} \right)} \notag \\
& \leqslant \exp \left( { - {N^{\kappa \zeta }}} \right)\sum\limits_{k \in {\mathbb{Z}^d}\setminus \left\{ 0 \right\}} g\left( {{{\left\| k \right\|}_{{\ell ^1}}}} \right)  \notag \\
&\lesssim \exp \left( { - {N^{\kappa \zeta }}} \right),\label{FJ2} 
\end{align}
where the following fact is used in \eqref{FJ2}, with its last inequality following from the earlier assumption on $ g $ in \eqref{GTJ}:
\[\sum\limits_{k \in {\mathbb{Z}^d}\setminus \left\{ 0 \right\}} {g\left( {{{\left\| k \right\|}_{{\ell ^1}}}} \right)}  \lesssim \int_{{[1,+\infty)^d}} {g\left( {{{\left\| r \right\|}_{{\ell ^1}}}} \right)d{r_1} \cdots d{r_d}}  \lesssim \int_1^{ + \infty } {{s^{d - 1}}g\left( s \right)ds}  <  + \infty .\]

For the sum over $\Lambda _{\alpha ,N}^ >$, we adopt a different strategy: we fully utilize the nonresonance with respect to $ N $ while treating regularity in a deliberately loose manner. To be more precise, we derive from Lemma \ref{GJGJ} that 
\begin{align}
\sum\limits_{k \in \Lambda _{\alpha ,N}^ > } {{{\left\| {{f_k}} \right\|}_{\mathscr{B}}}{\left| S_{N,k}(\alpha) \right|}}  &\leqslant \sum\limits_{k \in \Lambda _{\alpha ,N}^ > } {{{\left\| {{f_k}} \right\|}_{\mathscr{B}}}\exp \left( { - c{{\left( {N{\rm dist}\left( {\left\langle {k,\alpha } \right\rangle ,\mathbb{Z}} \right)} \right)}^{\frac{1}{{{\beta _{p,q}}}}}}} \right)} \notag \\
& \leqslant \left( {\sum\limits_{k \in {\mathbb{Z}^d}} {{{\left\| {{f_k}} \right\|}_{\mathscr{B}}}} } \right) \cdot \mathop {\sup }\limits_{k \in \Lambda _{\alpha ,N}^ > } \exp \left( { - c{{\left( {N{\rm dist}\left( {\left\langle {k,\alpha } \right\rangle ,\mathbb{Z}} \right)} \right)}^{\frac{1}{{{\beta _{p,q}}}}}}} \right) \notag \\
&  \lesssim \exp \left( { - c{N^{\frac{{1 - \zeta }}{{{\beta _{p,q}}}}}}} \right). \label{FJ3}
\end{align}
It is worth noting that Lemma  \ref{GJGJ} applies here because $ N{\rm dist}\left( {\left\langle {k,\alpha } \right\rangle ,\mathbb{Z}} \right) \gg 1 $. Conversely, it fails in the preceding case where $ k\in\Lambda_{\alpha,N}^{\leqslant} $.

As for the sum over $ k \in \Upsilon _\alpha ^ > $,  similar to the strategy in \eqref{FJ3}, we have 
\begin{align}
 \sum\limits_{k \in \Upsilon _\alpha ^ > } {{{\left\| {{f_k}} \right\|}_{\mathscr{B}}}{\left| S_{N,k}(\alpha) \right|}} & \leqslant \sum\limits_{k \in \Upsilon _\alpha ^ > } {{{\left\| {{f_k}} \right\|}_{\mathscr{B}}}\exp \left( { - c{{\left( {N{\rm dist}\left( {\left\langle {k,\alpha } \right\rangle ,\mathbb{Z}} \right)} \right)}^{\frac{1}{{{\beta _{p,q}}}}}}} \right)} \notag \\
& \lesssim \left(\sum\limits_{k \in {\mathbb{Z}^d} } {{{\left\| {{f_k}} \right\|}_{\mathscr{B}}}}\right)  \cdot \left(\mathop {\sup }\limits_{k \in \Upsilon _\alpha ^ > } \exp \left( { - c{{\left( {N{\rm dist}\left( {\left\langle {k,\alpha } \right\rangle ,\mathbb{Z}} \right)} \right)}^{\frac{1}{{{\beta _{p,q}}}}}}} \right)\right)\notag \\
& \lesssim \exp \left( { -c {{{\epsilon}^{  \frac{1}{{{\beta _{p,q}}}}}}}{N^{\frac{1}{{{\beta _{p,q}}}}}}} \right).\label{FJ4}
\end{align}
The estimate here, however, is not as loose as the one obtained in \eqref{FJ3}, due to the low regularity in the construction of $ f $.

Finally, we are in a position to derive the quantitative uniform exponential convergence of the weighted Birkhoff averages.
In light of our earlier choice of $ p , q ,\kappa$ and $ \zeta $, it is evident that 
\begin{equation}\label{GY}
	\min \left\{ {\frac{1}{{{\beta _{p,q}}}},\frac{{1 - \zeta }}{{{\beta _{p,q}}}},\kappa \zeta } \right\} > \sigma .
\end{equation}
Therefore, utilizing the decomposition property
\[\Lambda _{\alpha ,N}^ \leqslant  \cup \Lambda _{\alpha ,N}^ >  \cup \Upsilon _\alpha ^ >  = \Upsilon _\alpha ^ \leqslant  \cup \Upsilon _\alpha ^ >  =  {\mathbb{Z}^d}\setminus \left\{ 0 \right\},\quad  N \gg 1\]
and combining  \eqref{FJ1}, \eqref{FJ2}, \eqref{FJ3}, \eqref{FJ4} and \eqref{GY}, we finally arrive at
\begin{align}
 \mathop {\sup }\limits_{x \in \mathbb{T}^d} {\left\| {\sum\limits_{k \in {\mathbb{Z}^d}\setminus \left\{ 0 \right\}} {{f_k}{e^{2\pi i\left\langle {k,x} \right\rangle }}{ S_{N,k}(\alpha)  } } } \right\|_\mathscr{B}} 
  &\leqslant  \sum\limits_{k \in \Lambda _{\alpha ,N}^ \leqslant } { + \sum\limits_{k \in \Lambda _{\alpha ,N}^ > } { + \sum\limits_{k \in\Upsilon _\alpha ^ > } {{{\left\| {{f_k}} \right\|}_\mathscr{B}}{\left| S_{N,k}(\alpha) \right|}} } } \notag \\
& \lesssim   \exp \left( { - {N^{\kappa \zeta }}} \right) + \exp \left( { - c{N^{\frac{{1 - \zeta }}{{{\beta _{p,q}}}}}}} \right) + \exp \left( { -c {\epsilon^{  \frac{1}{{{\beta _{p,q}}}}}}{N^{\frac{1}{{{\beta _{p,q}}}}}}} \right) \notag \\
	& \lesssim  \exp \left( { -\widetilde{c} {N^\sigma }} \right),\quad  N \gg 1,\label{ZUIHOU}
\end{align}
provided with some universal constant $ \widetilde{c}>0 $.
This completes the proof of Theorem \ref{RET1}, since taking logarithms on both sides of \eqref{ZUIHOU} yields the desired estimate as
\[\mathop {\lim \sup }\limits_{N \to  + \infty } {\frac{1}{{{N^\sigma }}}}\log \left( {\mathop {\sup }\limits_{x \in \mathbb{T}^d} {{\left\| {\sum\limits_{n = 0}^{N - 1} {{w_{n,N}}f\left( {{{\mathscr{T}}_\alpha^n}x} \right)}  - \int_{\mathbb{T}^d} {fd\mu } } \right\|}_{\mathscr{B}}}} \right) =-\widetilde{c}<0.\]
 \section*{Acknowledgements} 
 The author would like to thank Prof. Yong Li, whose continuous support and constructive feedback greatly improved this work. The author is also grateful to the editor and the anonymous referee for the valuable suggestions and comments, which led to significant improvements in the manuscript. Z. Tong  was supported by the China Postdoctoral Science Foundation (Grant No. 2025M783102).

\end{document}